\documentclass[reqno]{amsart}
\usepackage{amsmath}
\usepackage{amsthm}
\usepackage{amssymb}

\usepackage[utf8]{inputenc}
\usepackage[T1]{fontenc}
\usepackage{lmodern}
\usepackage{eucal}
\pdfoutput=1 
\usepackage{microtype}
\usepackage{url}
\usepackage{enumitem}
\usepackage{tikz-cd}
\usepackage{longtable,booktabs}
\usepackage{array}

\usepackage[letterpaper, left = 1.5in, right = 1.5in, top = 1.5in, bottom = 1.5in]{geometry}

\swapnumbers

\usepackage{hyperref}

\usepackage{xcolor}
\hypersetup{
    colorlinks,
    linkcolor={red!50!black},
    citecolor={blue!50!black},
    urlcolor={blue!80!black}
}

\usepackage[capitalise,nameinlink]{cleveref}
\crefname{subsection}{Subsection}{Subsections}
\crefname{subsubsection}{Subsubsection}{Subsubsections}

\theoremstyle{definition}
\newtheorem{theorem}{Theorem}[section]
\newtheorem{defn}[theorem]{Definition}

\newtheorem{ex}[theorem]{Example}
\newtheorem{cor}[theorem]{Corollary}
\newtheorem{lemma}[theorem]{Lemma}
\newtheorem{prop}[theorem]{Proposition}
\newtheorem{rmk}[theorem]{Remark}

\newtheorem*{rmk*}{Remark}
\newtheorem*{ex*}{Example}
\newtheorem*{theorem*}{Theorem}
\newtheorem*{defn*}{Definition}

\newcommand{\bbZ}{\mathbb{Z}}

\newcommand{\bbE}{\mathbb{E}}

\newcommand{\bbR}{\mathbb{R}}

\newcommand{\bbC}{\mathbb{C}}

\newcommand{\Sp}{\mathcal{S}\mathrm{p}}

\newcommand{\Ext}{\operatorname{Ext}}

\renewcommand{\ker}{\operatorname{Ker}}

\newcommand{\Sq}{\mathrm{Sq}}

\newcommand{\tr}{\operatorname{tr}}

\newcommand{\TMF}{\operatorname{TMF}}

\newcommand{\kappabar}{\ol{\kappa}}

\newcommand{\h}{\mathrm{h}}

\newcommand{\bs}{{-}}

\newcommand{\ol}{\overline}

\newcommand\xqed[1]{%
  \leavevmode\unskip\penalty9999 \hbox{}\nobreak\hfill
  \quad\hbox{#1}}
\newcommand\tqed{\xqed{$\triangleleft$}}

\makeatletter
\DeclareRobustCommand{\tvdots}{%
  \vbox{\baselineskip4\p@\lineskiplimit\z@\kern0\p@\hbox{.}\hbox{.}\hbox{.}}}
\makeatother

\makeatletter
\@namedef{subjclassname@2020}{\textup{2020} Mathematics Subject Classification}
\makeatother

\begin{document}

\title[The Real-oriented cohomology of infinite stunted projective spaces]{The Real-oriented cohomology \\
 of infinite stunted projective spaces}

\author{William Balderrama}


\begin{abstract}
Let $E\mathbb{R}$ be an even-periodic Real Landweber exact $C_2$-spectrum, and $ER$ be its spectrum of fixed points. We compute the $ER$-cohomology of the infinite stunted real projective spectra $P_j$. These cohomology groups combine to form the $RO(C_2)$-graded coefficient ring of the $C_2$-spectrum $b(ER) = F(EC_{2+},i_\ast ER)$, which we show is related to $E\mathbb{R}$ by a cofiber sequence $\Sigma^\sigma b(ER)\rightarrow b(ER)\rightarrow E\mathbb{R}$. We illustrate our description of $\pi_\star b(ER)$ with the computation of some $ER$-based Mahowald invariants.
\end{abstract}

\maketitle


\section{Introduction}

The spectrum $MU$ of complex cobordism plays a central role in both our conceptual and computational understanding of stable homotopy theory. In \cite{landweber1968conjugations}, Landweber introduced what is now known as the $C_2$-equivariant spectrum $M\bbR$ of Real bordism, with underlying spectrum $MU$ and fixed points $MR = MU^{\h C_2}$ the homotopy fixed points for the action of $C_2$ on $MU$ by complex conjugation. Work of Araki \cite{araki1979orientations}, Hu--Kriz \cite{hukriz2001real}, and others, has shown that essentially all of the theory of complex-oriented homotopy theory may be carried out in the $C_2$-equivariant setting with $M\bbR$ in place of $MU$, leading to the rich subject of Real-oriented homotopy theory. This subject has seen extensive study over the past two decades, with a notable increase in interest following Hill--Hopkins--Ravenel's use of $M\bbR$ to resolve the Kervaire invariant one problem \cite{hillhopkinsravenel2016nonexistence}.

There are Real analogues of most familiar complex-oriented cohomology theories. An important family of examples is given by the Real Johnson--Wilson theories $E\bbR(n)$, refining the usual Johnson--Wilson theories $E(n)$. These theories are Landweber flat over $M\bbR$, in the sense that they are $M\bbR$-modules and satisfy
\[
E\bbR(n)_\star X \cong E\bbR(n)_\star \otimes_{M\bbR_\star} M\bbR_\star X
\]
for any $C_2$-spectrum $X$. The fixed points $ER(n) = E\bbR(n)^{C_2} = E(n)^{\h C_2}$ are nonequivariant cohomology theories that are interesting in their own right; for example, $ER(1) \simeq KO_{(2)}$, and $ER(2)$ is a variant of $\TMF_0(3)_{(2)}$. One may regard the descent from $E(n)$ to $ER(n)$ as encoding a portion of the $E(n)$-based Adams--Novikov spectral sequence, and accordingly each $ER(n)$ detects infinite families in $\pi_\ast S$.

There is in general a tradeoff between the richness of a homology theory and the ease with which it may be computed. Kitchloo, Lorman, and Wilson have carried out extensive computations with $ER(n)$-theory \cite{kitchloowilson2007hopf, kitchloowilson2015ern, lorman2016real, kitchloolormanwilson2017landweber, kitchloolormanwilson2018er2}, and their program has shown that these theories strike a very pleasant balance between richness and computability. Computations of $ER(2)^\ast \bbR P^n$ in particular have been applied to the nonimmersion problem for real projective spaces, with computations for $n = 2k$ in \cite{kitchloowilson2008second}, $n = 16k+1$ in \cite{kitchloowilson2008secondii}, and $n = 16k+9$ by Banerjee in \cite{banerjee2010real}.

This paper contributes to the above story. Let $E\bbR$ be a Real Landweber exact $C_2$-spectrum in the sense of Hill--Meier \cite[Section 3.2]{hillmeier2017c2}; we take this to include the assumption that $E\bbR$ is strongly even. Write $E$ for the underlying spectrum of $E\bbR$ and $ER = E\bbR^{C_2} = E^{\h C_2}$ for its fixed points. Suppose moreover that $E\bbR$ is even-periodic, in the sense that $\pi_{1+\sigma}E\bbR$ contains a unit. This is equivalent to asking that the $M\bbR$-orientation of $E\bbR$ extends to an $MP\bbR$-orientation, where 
\[
MP\bbR\simeq \bigoplus_{n\in\bbZ}\Sigma^{n(1+\sigma)}M\bbR
\]
is the Real analogue of $2$-periodic complex cobordism.

The primary goal of this paper is to compute the $ER$-cohomology of the infinite stunted projective spectra $P_j$. When $j > 0$, these are the spaces
\[
P_j = \bbR P^\infty / \bbR P^{j-1};
\]
in general, $P_j$ is the Thom spectrum of $j\sigma$, where $\sigma$ is the sign representation of $C_2$ regarded as a vector bundle over $BC_2 = \bbR P^\infty$. The cohomology $ER^\ast P_\ast$ is of interest for at least a few reasons: first, it is one long exact sequence away from the groups $ER^\ast \bbR P^j$, which have so far only been studied at heights $\leq 2$; second, there are $C_2$-equivariant Hurewicz maps $\pi_{c+w\sigma}S_{C_2}\rightarrow ER^{-c}P_w$, which are at least as nontrivial as the nonequivariant Hurewicz maps for $ER$; third, there is an interesting interplay between the $C_2$ appearing in $ER\simeq E^{\h C_2}$ and the $C_2$ appearing in $ER^\ast (P_w) \simeq ER^\ast (S^{w\sigma}_{\h C_2})$ which sheds some light on the nature of the $C_2$-spectrum $E\bbR$.

We record the basic properties of $ER$ in \cref{sec:even}. In particular, $\pi_0 ER \cong \pi_0 E$, the torsion in $\pi_\ast ER$ is supported on a single class $x\in \pi_1 ER$, there is a cofiber sequence $\Sigma ER\xrightarrow{x}ER\rightarrow E$, and the $x$-Bockstein spectral sequence for $ER$-cohomology agrees with the homotopy fixed point spectral sequence (HFPSS) from the $E_2$-page on.

Write $b(ER) = F(EC_{2+},i_\ast ER)$ for the Borel $C_2$-spectrum on $ER$ with trivial $C_2$-action. This satisfies $\pi_{c+w\sigma}b(ER) = ER^{-c}P_w$, and we shall compute $ER^\ast P_\ast$ using the $x$-Bockstein spectral sequence
\[
\pi_\star b(E)[x]\Rightarrow \pi_\star b(ER).
\]
This concludes an investigation we began in \cite{balderrama2021c2}. There, we computed the HFPSS $H^\ast(C_2;\pi_\star b(KU_2^\wedge))\Rightarrow\pi_\star b(KO_2^\wedge)$ as a step in our description of the $C_2$-equivariant $K(1)$-local sphere. At the time, we were able to put the $E_2$-page into a more general context by computing $H^\ast(C_2;\pi_\star b(E))$ for more general even-periodic Landweber exact spectra $E$, but had no information about possible higher differentials. In this paper, we carry out the rest of the computation for Real-oriented $E$. The results are summarized in \cref{sec:summary} below.

\begin{rmk}
The reader may observe that by restricting to even-periodic spectra, we have ruled out the real Johnson--Wilson theories $ER(n)$ for $n \geq 2$. However, any Real Landweber exact $C_2$-spectrum $E\bbR$ is a summand of the even-periodic theory $\bigoplus_{n\in\bbZ}\Sigma^{n(1+\sigma)}E\bbR$, so no real information has been lost. A more subtle point is that implicit in the definition of Real Landweber exactness is the assumption that $E\bbR$ is a ring up to homotopy, and it is not known whether $E\bbR(n)$ always satisfies this. However, the partial multiplicative structure given in \cite{kitchloolormanwilson2018multiplicative} is sufficient for our computation to apply to $2$-periodic $ER(n)$-theory.
\tqed
\end{rmk}

\subsection*{Acknowledgments}

We thank Hood Chatham for an enlightening conversation highlighting the role of Borel completeness in \cref{thm:xicofib}.

\section{Summary}\label{sec:summary}

We now describe our results. We start with the following, which serves as the linchpin for our computation of $\pi_\star b(ER)$. Write $\rho\in\pi_{-\sigma}S_{C_2}$ for the Euler class of the sign representation and $\tau^{-2}\in \pi_{2\sigma-2}b(E)$ for the Thom class of $2\sigma = \bbC \otimes \sigma$. These classes are sometimes denoted $a_\sigma$ and $u_{\sigma}^{-2}$, but we will reserve those symbols for $E\bbR$ and $C_{2+}\otimes i_\ast ER$.
Write $u\in \pi_2 E$ for the chosen unit, and set
\[
\xi = \rho\tau^{-2} u \in \pi_\sigma b(E).
\]

\begin{theorem}[\cref{sec:compare}]\label{thm:xicofib}
The class $\xi$ is a permanent cycle in the $x$-Bockstein spectral sequence, detecting a lift of $x$. Moreover, there is a cofiber sequence
\begin{equation}\label{eq:wood}\begin{tikzcd}
\Sigma^\sigma b(ER)\ar[r,"\xi"]&b(ER)\ar[r]&E\bbR
\end{tikzcd}\end{equation}
of $C_2$-spectra.
\tqed
\end{theorem}

This cofiber sequence is a twisted form of the standard cofiber sequence
\begin{equation}\label{eq:standard}\begin{tikzcd}
\Sigma^{-\sigma}b(ER)\ar[r,"\rho"]&b(ER)\ar[r]&C_{2+}\otimes i_\ast ER
\end{tikzcd}.\end{equation}

\begin{ex}
When $E = KU$, one can identify $b(ER) = F(EC_{2+},KO_{C_2})$ and $E\bbR = K\bbR$, and $\xi = \pm \eta_{C_2}$ is the $C_2$-equivariant Hopf map. In this case, \cref{thm:xicofib} recovers the Real Wood cofibering $KO_{C_2}/(\eta_{C_2})\simeq K\bbR$ (cf.\ \cite[Proposition 10.13]{guillouhillisaksenravenel2020cohomology}).
\tqed
\end{ex}

To show that $\xi$ is a permanent cycle detecting a lift of $x$, we first reduce to the universal case $E = MP$, then show that this is the only possibility compatible with norms on $b(MPR)$. Given this, the cofiber sequence of \cref{eq:wood} is a mostly formal consequence of \cref{eq:standard} and the fact that $\xi$ differs from $\rho$ by a unit in $\pi_\star b(E)$.

We now describe $\pi_\star b(ER)$. We start by fixing some notation for $\pi_\star b(E)$. Write $[2](z)\in E_0[[z]]$ for the $2$-series of the formal group law of $E$, and write $u_n \in E_0$ for the elements corresponding to the usual $v_n \in \pi_{2(2^n-1)}E$ by $u_n = u^{-(2^n-1)}v_n$. We may find series $h_n(z) \in E_0[[z]]$ for $n\geq 0$, of the form $h_n(z) = u_n + O(z)$ and satisfying
\[
[2](z) = z h_0(z),\qquad h_n(z)\equiv u_n + z^{2^n}h_{n+1}(z)\pmod{u_0,\ldots,u_{n-1}};
\]
Note in particular
\[
[2](z) \equiv z^{2^n}h_n(z)\pmod{u_0,\ldots,u_{n-1}}.
\]
We now specialize to $\pi_\star b(E)$. Set
\[
z = \rho \xi = \rho^2\tau^{-2} u,\qquad h_n = h_n(z),\qquad w_n = \rho^{2^{n+1}}h_{n+1}\equiv \tau^{2^{n+1}}u^{-2^n}(h_n-u_n),
\]
the last congruence being modulo $(u_0,\ldots,u_{n-1})$. We abbreviate $h = h_0$. This is the transfer element in $\pi_0 b(E) = E^0 BC_{2}$, and we have
\[
\pi_0 b(E) = \frac{E_0[[z]]}{([2](z))},\qquad \pi_\star b(E) = \frac{E_0[\rho,\tau^{\pm 2},u^{\pm 1}]_\rho^\wedge}{(\rho\cdot h)},
\]
see for instance \cite[Section 2.1]{balderrama2021c2}.
\begin{theorem}[\cref{sec:bock}]\label{thm:ber}
Define the subring $Z\subset\pi_\star b(E)$ by
\[
Z = E_0(\rho,\xi,\tau^{2^{n+2}l}u^{2^{n+1}k}u_n, \tau^{2^{n+1}(2l+1)}u^{2^{n+1}k}h_n:n\geq 0;k,l\in\bbZ)_{(\rho,\xi)}^\wedge\subset \pi_\star b(E),
\]
and let $B\subset Z[x]$ be the ideal generated by the elements
\begin{gather*}
\tau^{2^{n+2}l}u^{2^{n+1}k}u_n\cdot x^{2^{n+1}-1},\qquad \tau^{2^{n+1}(2l+1)}u^{2^{n+1}k}h_n\cdot x^{2^{n+1}-1},\qquad \tau^{2^{n+2}l}u^{2^{n+1}k}w_n\cdot x^{2^{n+1}-1}
\end{gather*}
for $n\geq 0$ and $k,l\in\bbZ$. Then $Z[x]/B$ is the $x$-adic associated graded of $\pi_\star b(ER)$.
\tqed
\end{theorem}

\begin{rmk}
In integer degrees, $\pi_\ast b(ER)$ is very simply described:
\[
\pi_\ast b(ER) \cong ER_\ast[[z]]/([2](z)),
\]
see \cref{cor:integer}. This does not require the full computation of $\pi_\star b(ER)$, and follows as soon as one knows that $\xi$ is a permanent cycle. In particular, $\pi_0 b(ER)\cong E^0 BC_2$. To get a feeling for $\pi_\star b(ER)$ outside integer degrees, the reader may wish to peruse the tables in \cref{rmk:table}, which list $\pi_0 b(ER)$-module generators for the groups $\pi_{c+w\sigma} b(ER)$ in a range.
\tqed
\end{rmk}

\begin{rmk}
Implicit in \cref{thm:ber} is the fact that $\tau^{2^{n+2}l}u^{2^{n+1}k}w_n \in Z$ for $n\geq 0$ and $k,l\in\bbZ$. In particular, 
\begin{align*}
\tau^{2^{n+2}(2l+1)}u^{2^{n+2}k}w_n &= \rho^{2^{n+1}}\cdot \tau^{2^{n+2}(2l+1)}u^{2^{n+2}k}h_{n+1}, \\
\tau^{2^{n+3}l}u^{2^{n+1}(2k+1)} w_n &= \xi^{2^{n+1}}\cdot \tau^{2^{n+2}(2l+1)}u^{2^{n+2}k}h_{n+1}, \\
\tau^{2^{n+3}l}u^{2^{n+2}k}w_n &\equiv \rho^{2^{n+1}}\cdot \tau^{2^{n+3}l}u^{2^{n+2}k}u_{n+1} + \xi^{2^{n+1}}\cdot \tau^{2^{n+3}l}u^{2^{n+2}k}w_{n+1},\\
\tau^{2^{n+2}(2l-1)}u^{2^{n+1}(2k+1)}w_n &\equiv \xi^{2^{n+1}}\cdot \tau^{2^{n+3}l}u^{2^{n+2}k}u_{n+1} + \rho^{2^{n+1}}\cdot\tau^{2^{n+3}(l-1)}u^{2^{n+2}(k+1)}w_{n+1},
\end{align*}
where the last two formulas hold mod $u_0,\ldots,u_n$.
\tqed
\end{rmk}

The ring $Z[x]/B$ is the $E_\infty$-page of the $x$-Bockstein spectral sequence for $\pi_\star b(ER)$, obtained after running differentials which are generated by
\[
d_{2^{n+1}-1}(u^{2^n}) = u_n x^{2^{n+1}-1},\qquad d_{2^{n+1}-1}(\tau^{2^{n+1}}) = -w_n x^{2^{n+1}-1}.
\]
The differentials on $u^{2^n}$ appear in the $x$-Bockstein spectral sequence for $\pi_\ast ER$, and are consequences of Hu and Kriz's computation of $\pi_\star M\bbR$ \cite{hukriz2001real}, as we review in \cref{sec:even}. The differentials on $\tau^{2^{n+1}}$ are the core of our computation. These differentials turn out to be forced by the permanent cycle $\xi = \rho\tau^{-2}u$, by a Leibniz rule argument based on $d_{2^{n+1}-1}(\xi^{2^n}) = 0$. This argument would not be possible if one tried to compute each $ER^\ast P_j$ individually, and illustrates the strength of using the $C_2$-spectrum $b(ER)$ as a tool for packaging information about the cohomology of all stunted projective spectra into one object.

One might also try to understand $\pi_\star b(ER)$ through the $\rho$-Bockstein or the $\xi$-Bockstein spectral sequences. Using the cofiber sequences \cref{eq:standard} and \cref{eq:wood}, these are of signature
\[
\pi_\star(C_{2+} \otimes i_\ast ER)[\rho]\Rightarrow \pi_\star b(ER),\qquad \pi_\star E\bbR[\xi]\Rightarrow \pi_\star b(ER).
\]
Here, $\pi_\star (C_{2+}\otimes i_\ast ER)\cong \pi_\ast ER[u_\sigma^{\pm 1}]$ with $|u_\sigma| = 1-\sigma$, and in degrees $\ast+w\sigma$ the $\rho$-Bockstein spectral sequence is exactly the Atiyah--Hirzebruch spectral sequence for $ER^\ast P_w$ based on the standard cell structure of $P_w$. By construction, the differentials in these spectral sequences are controlled by the boundary maps
\[
\tr(u_\sigma^{-1}\cdot\bs)\colon \pi_{\star+1-\sigma}(C_{2+}\otimes i_\ast ER)\rightarrow \pi_{\star} b(ER),\qquad \partial\colon \pi_{\star+1+\sigma} E\bbR\rightarrow \pi_\star b(ER)
\]
for the cofiber sequences \cref{eq:standard} and \cref{eq:wood}. This first boundary map is exactly the transfer for the $C_2$-spectrum $b(ER)$. Although we do not know whether it is is feasible to compute either the $\rho$-Bockstein or $\xi$-Bockstein spectral sequence directly, we can use our computation of $\pi_\star b(ER)$ to deduce the following.

Write $\ol{u} \in \pi_{1+\sigma}E\bbR$ for the invertible element guaranteed by the $MP\bbR$-orientation of $E\bbR$.

\begin{theorem}[\cref{sec:tr}]\label{thm:trpartial}
The above transfer and boundary maps satisfy
\begin{align*}
\tr(u_\sigma^{-1}\cdot u_\sigma^{2^n(2k+1)}) &= \rho^{2^n-1}\tau^{2^{n+1}k}h_n x^{2^n-1}+O(x^{2^n}),\\
\partial(\ol{u}^{2^n(2k+1)}) &= \xi^{2^n-1}\tau^{-2^{n+1}k}u^{2^{n+1}k}h_nx^{2^n-1}+O(x^{2^n})
\end{align*}
for $n\geq 0$ and $k\in\bbZ$.
\tqed
\end{theorem}

\pagebreak[1]

The error terms here are necessary as the classes $\tau^{2^{n+1}k}h_n$ and $\tau^{-2^{n+1}k}u^{2^{n+1}k}h_n$ have only been defined mod $x$. It is amusing to observe that \cref{thm:trpartial} produces elements of arbitrarily high $x$-adic filtration in the $C_2$-equivariant Hurewicz image of $b(MPR)$; as far as we know, such families have not yet been constructed in the nonequivariant Hurewicz image of $MPR$.

\cref{thm:ber} does not quite describe the ring $\pi_\star b(ER)$, but only its $x$-adic associated graded $Z[x]/B$. The latter is a good approximation to the former, particularly when compared to the $\rho$-adic and $\xi$-adic associated graded rings, where the classes $\rho$ and $\xi$ appear as simple $2$-torsion classes. Still, taking the $x$-adic associated graded does kill some information, and it seems to be a subtle problem to completely reconstruct the ring $\pi_\star b(ER)$. Although we shall not completely resolve this, we do discuss where to find hidden $\rho$ and $\xi$-extensions. The importance of $\rho$-extensions is clear: as
\[
\pi_{c+w\sigma}(b(ER)/(\rho^{m})) = ER^{-c} (P^{w-1}_{w-m}),
\]
one must understand the action of $\rho$ if one wishes to extract information about the $ER$-cohomology of finite projective spaces. The importance of $\xi$-extensions is clear from the perspective of $C_2$-equivariant homotopy theory: just as important classes in the Hurewicz image of $ER$ are supported on $x$, important classes in the $C_2$-equivariant Hurewicz image of $b(ER)$ are supported on $\xi$, such as the equivariant Hopf fibrations $\eta_{C_2}$, $\nu_{C_2}$, and $\sigma_{C_2}$ detected in $\pi_\star b(ER)$ by $h_1\xi$, $h_2\xi^2 x$, and $h_3\xi^4 x^3$ respectively, and so the action of $\xi$ gives information about the behavior of these elements. The cofiber sequences \cref{eq:wood} and \cref{eq:standard} give information about $\rho$ and $\xi$-extensions, leading to the following.

\pagebreak[1]

\begin{samepage}
\begin{theorem}[\cref{sec:ext}]
There are extensions
\begin{align*}
\rho\cdot \tau^{2(2^{n+1}k-r)}u^{2^{n+1}(2l+1)}h &= \left(\tau^{2^{n+2}k}u^{2^{n+2}l}h_{n+1} \xi^{2r-1}+O(\rho)\right)x^{2^{n+2}-2r}+O(x^{2^{n+2}-2r+1})\\
\xi\cdot\tau^{2(2^{n+1}k+r)}u^{2(2^n(2l+1)-r)}h &= \left(\tau^{2^{n+2}k}u^{2^{n+2}l}h_{n+1} \rho^{2r-1}+O(\xi)\right)x^{2^{n+2}-2r}+O(x^{2^{n+2}-2r+1})
\end{align*}
for $n\geq 0$, $k,l\in\bbZ$, and $1\leq r \leq 2^{n+1}-1$.
\tqed
\end{theorem}
\end{samepage}

As with \cref{thm:ber}, implicit in this theorem is the fact that the terms on the left and right do in fact live in $\pi_\star b(ER)$, for example $\tau^4 h = 2\tau^4 + \rho\xi\cdot\tau^4h_1$. The error terms are present to remind the reader that these are extensions and not products: to resolve them would require describing how to lift classes from $Z$ to $\pi_\star b(ER)$, and we shall not pursue this. In particular, if $k$ is even then the $h_{n+1}$ terms on the right may be replaced with $u_{n+1}$ without affecting the theorem statement.

This concludes our description of $\pi_\star b(ER)$. Although $\pi_\star b(ER)$ is complicated, it is not impossible to work with. We illustrate this in \cref{sec:root} by computing some $MPR$-based Mahowald invariants. Li--Shi--Wang--Xu \cite{lishiwangxu2019hurewicz} have shown that Real bordism detects the Hopf elements, Kervaire classes, and $\kappabar$ family. These are the elements in $\pi_\ast S$ detected in the classical Adams spectral sequence by the $\Sq^0$-families generated by $h_0$, $h_0^2$, and $g_1$. We compute the iterated $MPR$-based Mahowald invariants of $2$, $4$, and $\kappabar$, showing that they line up with these $\Sq^0$-families exactly.

\section{Even-periodic Real Landweber exact spectra}\label{sec:even}

We begin by recording some properties of $E\bbR$ and $ER$. The material of this section is essentially a translation to the even-periodic setting of familiar facts about the Real Johnson--Wilson theories. We would like to avoid confusion between elements of $\pi_\star b(ER)$ and $\pi_\star E\bbR$, so in this section we write $a_\sigma$ instead of $\rho$ for the Euler class of the sign representation, and use the symbol $u_\sigma$ for what would previously have been written $\tau$. In particular, these symbols have degrees
\[
|a_\sigma| = -\sigma,\qquad |u_\sigma| = 1-\sigma.
\]

Before considering $E\bbR$, we consider the $C_2$-HFPSS in general. There is a cofiber sequence
\begin{equation}\label{eq:rhocof}
\begin{tikzcd}
S^{-\sigma}\ar[r,"a_\sigma"]&S^0\ar[r]&C_{2+}
\end{tikzcd}.\end{equation}
Let $X$ be a $C_2$-spectrum. Then we may identify
\[
\pi_\star(C_{2+}\otimes X)\cong \pi_\ast^e X[u_\sigma^{\pm 1}],
\]
and so the $a_\sigma$-Bockstein spectral sequence for $X$ is of signature
\[
E_1 = \pi_\ast^e X[u_\sigma^{\pm 1},a_\sigma]\Rightarrow \pi_\star X,
\]
where $\pi_\ast^e X$ are the homotopy groups of the underlying spectrum of $X$. This spectral sequence converges conditionally to the homotopy groups of the $a_\sigma$-completion of $X$, which may be identified as its Borel completion $F(EC_{2+},X)$. Moreover we have the following fact, see for instance \cite[Lemma 4.8]{hillmeier2017c2}.

\begin{lemma}\label{lem:bockhfpss}
For any $C_2$-spectrum $X$, the $a_\sigma$-Bockstein spectral sequence for $X$ agrees with the HFPSS for $X$ from the $E_2$-page on.
\qed
\end{lemma}

The proof amounts to identifying the $a_\sigma$-Bockstein spectral sequence with the Borel cohomology spectral sequence induced by the standard cellular filtration of $EC_{2+}$. This identification leads to the following.

\begin{lemma}\label{lem:hfpssd1}
Let $X$ be a $C_2$-spectrum, and write $\psi^{-1}$ for the involution on $\pi_\ast^e X$. Then the $d_1$-differential in the $a_\sigma$-Bockstein spectral sequence for $X$ is given by
\[
d_1(\alpha u_\sigma^n a_\sigma^m) = (\alpha - (-1)^n \psi^{-1}(\alpha))u_\sigma^{n-1}a_\sigma^{m+1}
\]
for $\alpha\in \pi_\ast^e X$. In particular, if $X$ carries a product, then the differentials satisfy the Leibniz rule
\[
d_r(\alpha\beta) = d_r(\alpha)\beta+\psi^{-1}(\alpha)d_r(\beta)
\]
for $r\geq 1$, where the $\psi^{-1}$ may be omitted for $r\geq 2$.
\qed
\end{lemma}

Now let $E\bbR$ be as in the introduction: a strongly even and even-periodic and Real Landweber exact $C_2$-spectrum in the sense of \cite[Section 3.2]{hillmeier2017c2}, with underlying spectrum $E$. This set of assumptions means three things. First, $E\bbR$ is a homotopy commutative $C_2$-ring spectrum equipped with a multiplicative orientation $MP\bbR\rightarrow E\bbR$. In particular, there is an invertible element $\ol{u} \in \pi_{1+\sigma}E\bbR$ coming from the generator of the $n=1$ summand of $MP\bbR\simeq\bigoplus_{n\in\bbZ}\Sigma^{n(1+\sigma)}M\bbR$. Second, $\pi_0 E\bbR \cong \pi_0 E$, and in general
\[
E\bbR_\star X \cong E_0\otimes_{MP_0}MP\bbR_\star X
\]
for any $C_2$-spectrum $X$. Third, $\pi_{-1}E\bbR = 0$. Implicit in these is the fact that $MP\bbR$ itself satisfies these conditions. This is nontrivial, and follows from work of Hu and Kriz on Real cobordism \cite{hukriz2001real}. We recall the key calculation.

\begin{lemma}\label{lem:erbock}
The $C_2$-spectrum $E\bbR$ is Borel complete, with $a_\sigma$-Bockstein spectral sequence
\[
E_1 = E_0[\ol{u}^{\pm 1},u_\sigma^{\pm 1},a_\sigma]\Rightarrow \pi_\star E\bbR,
\]
where
\[
|\ol{u}| = 1+\sigma,\qquad |u_\sigma| = 1-\sigma,\qquad |a_\sigma| = -\sigma.
\]
The differentials are $E_0[\ol{u}^{\pm 1},a_\sigma]$-linear, and are generated by
\[
d_{2^{n+1}-1}(u_\sigma^{2^n}) = u_n \ol{u}^{2^n-1}a_\sigma^{2^{n+1}-1},
\]
where $u_n = u^{-(2^n-1)} v_n \in E_0$. In particular, $\pi_0 E\bbR = \pi_0 E$.
\end{lemma}
\begin{proof}
We first verify the given description of the $a_\sigma$-Bockstein spectral sequence. The $E_1$-page of the $a_\sigma$-Bockstein spectral sequence is given by $E_\ast[u_\sigma^{\pm 1},a_\sigma]$. To put this in the desired form, we set $\ol{u} = u_\sigma^{-1} u$ with $u\in \pi_2 E$ the unit; when $E = MP$, this generates the $n=1$ summand of $MP\bbR\simeq\bigoplus_{n\in\bbZ}\Sigma^{n(1+\sigma)}M\bbR$. As $E\bbR$ is Landweber exact over $MP\bbR$, the $a_\sigma$-Bockstein spectral sequence for $E\bbR$ is tensored down from the $a_\sigma$-Bockstein spectral sequence for $MP\bbR$, and here the computation is known by work of Hu and Kriz \cite{hukriz2001real}.

The $C_2$-spectrum $M\bbR$ is shown to be Borel complete in \cite[Theorem 4.1]{hukriz2001real}, and Landweber exactness extends the proof to $E\bbR$. By the Tate fracture square, $E\bbR$ is Borel complete if and only if the map $\Phi^{C_2}E\bbR\rightarrow E\bbR^{t C_2}$ is an equivalence, where $\Phi^{C_2}$ denotes the functor of geometric fixed points. Landweber exactness implies
\[
\pi_\ast \Phi^{C_2}E\bbR \cong E_0\otimes_{MP_0}\pi_\ast \Phi^{C_2}MP\bbR \cong E_0/(u_0,u_1,\ldots)[x^{\pm 1}],
\]
where $x = a_\sigma \ol{u} \in \pi_1 E\bbR$, the last identification coming from the equivalence $\Phi^{C_2}M\bbR\simeq MO$. This is exactly what one obtains computing $\pi_\ast E\bbR^{t C_2}$ by the Tate spectral sequence, which may itself be obtained from the above description of the $a_\sigma$-Bockstein spectral sequence by inverting $a_\sigma$. Thus $E\bbR$ is Borel complete as claimed.
\end{proof}

We now pass to the nonequivariant spectrum $ER = E\bbR^{C_2} \simeq E^{\h C_2}$. Note that $\pi_\ast ER$ is the portion of $\pi_\star E\bbR$ located in integer degrees, and write $x = a_\sigma \ol{u} \in \pi_1 ER$. We then have the following analogue of \cite{kitchloowilson2007fibrations} and \cite[Theorem 4.2]{kitchloowilson2008second}.

\pagebreak[1]

\begin{prop}\label{lem:cofib}
There is a cofiber sequence
\begin{center}\begin{tikzcd}
\Sigma ER\ar[r,"x"]&ER\ar[r]&E
\end{tikzcd},\end{center}
and thus for any spectrum $X$ an $x$-Bockstein spectral sequence
\[
E_1 = (E^\ast X)[x]\Rightarrow ER^\ast X,
\]
and this agrees with the HFPSS
\[
E_2 = H^\ast(C_2;E^\ast X)\Rightarrow ER^\ast X
\]
from the $E_2$-page on. Write $\psi^{-1}$ for the involution on $E^\ast X$. Then the $d_1$-differential in the $x$-Bockstein spectral sequence is given by
\[
d_1(\alpha) = (\alpha - \psi^{-1}(\alpha))x
\]
for $\alpha\in E^\ast X$, and the differentials satisfy the Leibniz rule
\[
d_r(\alpha\beta) = d_r(\alpha)\beta+\psi^{-1}(\alpha)d_r(\beta)
\]
for $r\geq 1$, where the $\psi^{-1}$ may be omitted for $r\geq 2$. All $x$-Bockstein differentials are $E_0$-linear, and when $X = S^0$ they are generated by
\[
d_{2^{n+1}-1}(u^{2^n}) = u_n x^{2^{n+1}-1}.
\]
\end{prop}
\begin{proof}
As $x = a_\sigma \ol{u}$ and $\ol{u}$ is invertible in $E\bbR$, \cref{eq:rhocof} implies that there is a cofiber sequence
\begin{center}\begin{tikzcd}
\Sigma E\bbR\ar[r,"x"]&E\bbR\ar[r]&C_{2+}\otimes E\bbR
\end{tikzcd}\end{center}
of $C_2$-spectra. Passing to fixed points yields the corresponding cofiber sequence for $ER$. The remaining facts follow from the previous lemmas.
\end{proof}

\begin{rmk}
The following depicts the $E_\infty$ page of the $x$-Bockstein spectral sequence for $ER$:

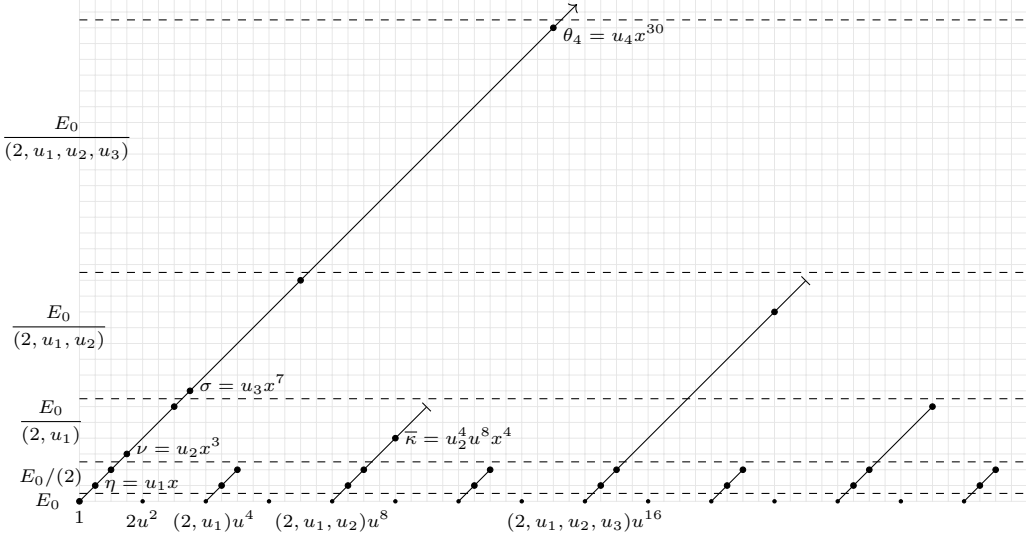
\begin{center}
\begin{tikzpicture}[scale = 0.21, font = \scriptsize]
\draw[step = 1cm, very thin,opacity = 0.1] (0,0) grid (60,32);
\draw [->] (0,0) -- (31.5,31.5);
\filldraw (0,0) circle (5pt) node [below] {$1$};
\filldraw (1,1) circle (5pt) node [right] {$\eta = u_1 x$};
\filldraw (2,2) circle (5pt);
\filldraw (3,3) circle (5pt) node [right, yshift=2.5pt] {$\nu = u_2x^3$};
\filldraw (6,6) circle (5pt);
\filldraw (7,7) circle (5pt) node [right, yshift=2.5pt] {$\sigma = u_3x^7$};
\filldraw (14,14) circle (5pt);
\filldraw (30,30) circle (5pt) node [right, yshift=-2.5pt] {$\theta_4 = u_4 x^{30}$};

\draw[dashed,opacity=1](0,0.5)--(60,0.5);
\draw[dashed,opacity=1](0,2.5)--(60,2.5);
\draw[dashed,opacity=1](0,6.5) -- (60,6.5);
\draw[dashed,opacity=1](0,14.5) -- (60,14.5);
\draw[dashed,opacity=1](0,30.5) -- (60,30.5);

\node[align = left] at (-2,0) {$E_0$};
\node[align = left] at (-1.8,1.5) {$E_0/(2)$};
\node[align = left] at (-1.8,5) {$\dfrac{E_0}{(2,u_1)}$};
\node[align = left] at (-1.3,11) {$\dfrac{E_0}{(2,u_1,u_2)}$};
\node[align = left] at (-0.8,23) {$\dfrac{E_0}{(2,u_1,u_2,u_3)}$};

\filldraw (4,0) circle (3pt) node [below] {$2u^2$};
\filldraw (8,0) circle (3pt) node [below, xshift=3pt] {$(2,u_1)u^4$};
\filldraw (12,0) circle (3pt);
\filldraw (16,0) circle (3pt) node [below] {$(2,u_1,u_2)u^8$};
\filldraw (20,0) circle (3pt);
\filldraw (24,0) circle (3pt);
\filldraw (28,0) circle (3pt);
\filldraw (32,0) circle (3pt) node [below] {$(2,u_1,u_2,u_3)u^{16}$};
\filldraw (36,0) circle (3pt);
\filldraw (40,0) circle (3pt);
\filldraw (44,0) circle (3pt);
\filldraw (48,0) circle (3pt);
\filldraw (52,0) circle (3pt);
\filldraw (56,0) circle (3pt);

\draw (8,0) -- (10,2);
\draw [-|] (16,0) -- (22,6);
\draw (24,0) -- (26,2);
\draw [-|] (32,0) -- (46,14);
\draw (40,0) -- (42,2);
\draw (48,0) -- (54,6);
\draw (56,0) -- (58,2);

\filldraw (9,1) circle (5pt);
\filldraw (10,2) circle (5pt);
\filldraw (17,1) circle (5pt);
\filldraw (18,2) circle (5pt);
\filldraw (25,1) circle (5pt);
\filldraw (26,2) circle (5pt);
\filldraw (33,1) circle (5pt);
\filldraw (34,2) circle (5pt);
\filldraw (41,1) circle (5pt);
\filldraw (42,2) circle (5pt);
\filldraw (49,1) circle (5pt);
\filldraw (50,2) circle (5pt);
\filldraw (57,1) circle (5pt);
\filldraw (58,2) circle (5pt);

\filldraw (20,4) circle (5pt) node [right] {$\overline{\kappa} = u_2^4 u^8x^4$};

\filldraw (44,12) circle (5pt);

\filldraw(54,6) circle (5pt);
\end{tikzpicture}
\end{center}

Here, the lines of slope $1$ depict $x$-towers. Everything in filtration $\geq 2^n-1$ is a module over $E_0/(u_0,\ldots,u_{n-1})$, and these regions are separated by dashed lines. The terms on the bottom describe the $0$-line $(\pi_\ast ER)/(x)$. For example, $(\pi_{16}ER)/(x) \subset \pi_{16} E \cong E_0\{u^8\}$ is the sub-$E_0$-module generated by $(2u^8,u_1u^8,u_2u^8)$; as $2u^8\cdot x = 0$ and $u_1u^8\cdot x^3 = 0$, the $x$-tower out of this is supported on $u_1u^8$ and $u_2u^8$, and on just $u_2u^8$ starting in filtration $3$. The solid circles in positive filtration indicate degrees where $\pi_\ast ER$ has nontrivial Hurewicz image, with some notable elements labeled (see \cref{sec:root}).

Similar charts appear in \cite[Section 6]{hahnshi2020real}.
\tqed
\end{rmk}

\section{Comparing \texorpdfstring{$b(ER)$ and $E\bbR$}{b(ER) and ER}}\label{sec:compare}

We are now in a position to consider \cref{thm:xicofib}. The first order of business is to identify $\xi = \rho\tau^{-2} u \in \pi_\sigma b(E)$ as a permanent cycle in the $x$-Bockstein spectral sequence for $\pi_\star b(ER)$. As $\rho\colon \pi_{\sigma}b(ER)\rightarrow \pi_0 b(ER) = ER^0 \bbR P^\infty$ is the inclusion of a summand, the fact that $\xi$ is a permanent cycle in the $x$-Bockstein spectral sequence for $\pi_\star b(ER)$ is predicted by Kitchloo--Wilson's computation of $ER(n)^\ast \bbR P^\infty$ \cite[Theorem 1.2]{kitchloowilson2008second}, see also \cite{kitchloolormanwilson2017landweber}. However, we take a different approach that sheds light on additional aspects of $\xi$.

Because the $x$-Bockstein spectral sequence for $\pi_\star b(ER)$ agrees with the HFPSS
\[
E_2 = H^\ast(C_2;\pi_\star b(E))\Rightarrow \pi_\star b(ER)
\]
from the $E_2$-page on, we can just as well work with the HFPSS in this section.

\begin{lemma}\label{lem:bpp}
We have
\[
\pi_\star b(E) = \frac{E_0[\rho,\tau^{\pm 2},u^{\pm 1}]_\rho^\wedge}{(\rho\cdot h)},
\]
and $C_2$ acts on $\pi_\star b(E)$ by the $E_0$-linear multiplicative involution $\psi^{-1}$ satisfying
\[
\psi^{-1}(\rho) = \rho,\qquad \psi^{-1}(u) = -u,\qquad \psi^{-1}(\tau^2) = \tau^2(h-1).
\]
In particular, $\xi$ is fixed under the action of $\psi^{-1}$.
\end{lemma}
\begin{proof}
The structure of $\pi_\star b(E)$ is as described in \cite[Section 2.1]{balderrama2021c2}. That $\psi^{-1}$ fixes $\xi$ follows immediately.
\end{proof}

\begin{prop}\label{prop:xipc}
The class $\xi$ is a permanent cycle in the HFPSS for $\pi_\star b(ER)$, detecting a lift of $x$.
\end{prop}
\begin{proof}
By assumption, $E\bbR$ is $MP\bbR$-oriented. As $\xi = \rho \tau^{-2} u$ lifts to $\pi_\sigma b(MP)$ and $x$ lifts to $\pi_1 MPR$, it suffices to prove the proposition with $ER$ replaced by $MPR$.

As $MP$ is an $\bbE_\infty$ ring, and complex conjugation acts on $MP$ by $\bbE_\infty$ automorphisms, there is a $C_2$-equivariant external squaring operation
\[
\Sq\colon \pi_n MP\rightarrow \pi_{n(1+\sigma)}b(MP).
\]
As $\Sq$ is additive modulo transfers and $\rho$ annihilates the transfer ideal, the composite $\rho \cdot \Sq$ is additive and so induces a map
\[
Q\colon H^1(C_2;\pi_{n+1}MP)\rightarrow H^1(C_2;\pi_{n(1+\sigma)+1}b(MP))
\]
in group cohomology. By \cite[Theorem 1.0.1]{balderrama2022total}, if $a\in H^1(C_2;\pi_{n+1}MP)$ is a permanent cycle detecting $\alpha\in \pi_n MPR$, then $Q(a)$ is a permanent cycle weakly detecting $\Sq(\alpha)\in \pi_{n(1+\sigma)}b(MPR)$. Now recall that $x$ represents the generator of $H^1(C_2;\bbZ\{u\})\cong\bbZ/(2)\subset H^1(C_2;\pi_2 MP)$. The $\bbE_\infty$ structure on periodic cobordism is such that $\Sq(u) = \tau^{-2}u^2$; see for instance the paragraph after \cite[Lemma 4.3]{andohopkinsstrickland2004sigma}, noting that $u$ and $\tau^{-2}u^2$ are the periodic Thom classes for $\bbC$ and $\bbC[C_2]$ respectively. Thus $Q(x) = \xi x$, and it follows that $\xi x$ detects $\Sq(x)$. As $\Sq(x)$ lifts $x^2$, this is only possible if $\xi$ is a permanent cycle detecting a lift of $x$ as claimed.
\end{proof}

The following corollary is not needed for \cref{thm:xicofib}, but will be useful later on in understanding the structure of $\pi_\star b(ER)$. It is a direct analogue of \cite[Theorem 1.2]{kitchloowilson2008second}.

\begin{cor}\label{cor:integer}
In integer degrees, we have
\[
\pi_\ast b(ER) \cong ER_\ast[[z]]/([2](z)),
\]
where $z = \rho \xi$. In particular, $\pi_\star b(ER)$ is a module over $\pi_0 b(ER) \cong E^0 BC_2\cong E_0[[z]]/([2](z))$.
\end{cor}
\begin{proof}
The $x$-Bockstein spectral sequence for $\pi_\ast b(ER)$ takes the form
\[
E_1 = \pi_\ast b(E)[x]\Rightarrow \pi_\ast b(ER).
\]
Recall that
\[
\pi_\ast b(E) \cong E_\ast[[z]]/([2](z)),\qquad z = \rho \xi.
\]
As $\rho$ and $\xi$ are permanent cycles, so is $z$. Thus the differentials in the $x$-Bockstein spectral sequence for $\pi_\ast b(ER)$ are induced by those for $\pi_\ast ER$, leading to the given description of $\pi_\ast b(ER)$.
\end{proof}

We now relate $b(ER)$ and $E\bbR$. These live in the full subcategory $\Sp^{BC_2}\subset\Sp^{C_2}$ of Borel complete $C_2$-spectra, equivalent to the category of spectra with $C_2$-action. The functor
\[
b\colon \Sp\rightarrow\Sp^{BC_2},\qquad b(X) = F(EC_{2+},i_\ast X)
\]
is the diagonal, endowing a spectrum with the trivial $C_2$-action. In particular, it is left adjoint to the functor of homotopy fixed points, and if $X\in\Sp^{BC_2}$ then the counit of this adjunction gives a canonical map
\[
b(X^{\h C_2})\rightarrow X.
\]
Specializing to $X = E\bbR$, we have the following.

\begin{theorem}\label{prop:xicof}
The canonical map $b(ER)\rightarrow E\bbR$ fits into a cofiber sequence
\begin{equation}\label{eq:xic}
\begin{tikzcd}
\Sigma^\sigma b(ER)\ar[r,"\xi"]&b(ER)\ar[r]&E\bbR
\end{tikzcd}
\end{equation}
of $C_2$-spectra.
\end{theorem}
\begin{proof}
As $E\bbR$ is strongly even, we have $\pi_\sigma E \bbR = 0$. As the maps in \cref{eq:xic} are $b(ER)$-linear, their composite must be null. As $b(ER)$ and $E\bbR$ are $x$-complete, it then suffices to show that \cref{eq:xic} is a cofiber sequence after coning off $x$. As $b(ER)/(x)\simeq b(E)$ and $E\bbR/(x)\simeq C_{2+}\otimes i_\ast E$, \cref{eq:xic} with $x$ coned off takes the form
\begin{center}
\begin{tikzcd}
\Sigma^\sigma b(E)\ar[r,"\xi"]&b(E)\ar[r]&C_{2+}\otimes i_\ast E
\end{tikzcd},
\end{center}
which is a cofiber sequence as now $\xi = \rho\cdot \tau^{-2}u$ differs from $\rho$ by a unit.
\end{proof}

\cref{thm:xicofib} follows by combining \cref{prop:xipc} and \cref{prop:xicof}.

\section{The Bockstein spectral sequence}\label{sec:bock}

We now compute the $x$-Bockstein spectral sequence
\[
\pi_\star b(E)[x]\Rightarrow\pi_\star b(ER).
\]
We maintain notation from the introduction. In particular, recall that $h_n$ and $w_n$ are defined in terms of the $2$-series of $E$ by specializing
\begin{gather*}
[2](z) = z h_0(z),\qquad h_n(z)\equiv u_n + z^{2^n}h_{n+1}(z)\pmod{u_0,\ldots,u_{n-1}},\\
w_n = \rho^{2^{n+1}}h_{n+1} \equiv \tau^{2^{n+1}}u^{-2^n}(h_n-u_n) \pmod{u_0,\ldots,u_{n-1}}
\end{gather*}
to $z = \rho \xi= \rho^2\tau^{-2}u$. As with $u_n$, these classes are well defined modulo $(u_0,\ldots,u_{n-1})$. We begin by describing what will be the cycles and boundaries of the $x$-Bockstein spectral sequence for $\pi_\star b(ER)$. Let $Z_{2^{n+1}-1}\subset\pi_\star b(E)$ be the subring
\[
E_0(\rho,\xi,u^{\pm 2^{n+1}},\tau^{2^{\pm n+2}},\tau^{2^{i+2}l}u^{2^{i+1}k}u_i,\tau^{2^{i+1}(2l+1)}u^{2^{i+1}k}h_i:0\leq i \leq n;k,l\in\bbZ)_{(\rho,\xi)}^\wedge,
\]
and let $B_{2^{n+1}-1}\subset Z_{2^{n+1}-1}[x]$ be the ideal generated by
\[
\tau^{2^{i+2}l}u^{2^{i+1}k}u_i\cdot x^{2^{i+1}-1},\qquad \tau^{2^{i+1}(2l+1)}u^{2^{i+1}k}h_i\cdot x^{2^{i+1}-1},\qquad \tau^{2^{i+2}l}u^{2^{i+1}k}w_i\cdot x^{2^{i+1}-1}
\]
for $0\leq i \leq n$ and $k,l\in\bbZ$. We also declare $Z_0 = \pi_\star b(E)$ and $B_0 = (0)$, and for $2^{n+1}-1 \leq r < 2^{n+2}-1$ we write $Z_r = Z_{2^{n+1}-1}$ and $B_r = Z_{2^{n+1}-1}$. Thus there are inclusions
\[
0 = B_0 \subset B_1\subset B_2\subset\cdots \subset Z_2[x]\subset Z_1[x]\subset Z_0[x] = \pi_\star b(E)[x].
\]

\begin{theorem}\label{thm:xbock}
The $x$-Bockstein spectral sequence for $\pi_\star b(ER)$ supports differentials
\[
d_{2^{n+1}-1}(u^{2^n}) = u_n x^{2^{n+1}-1},\qquad d_{2^{n+1}-1}(\tau^{2^{n+1}}) = -w_n x^{2^{n+1}-1},
\]
and we may identify $Z_r[x]$ and $B_r$ as its $r$-cycles and $r$-boundaries.
\end{theorem}
\begin{proof}
We proceed by induction, treating the inductive step first.

Let $n\geq 1$, and suppose we have verified $E_{2^n} \cong Z_{2^n-1}[x]/B_{2^n-1}$. In particular, $E_{2^n}$ is generated by the permanent cycles $\rho$ and $\xi$, the classes $\tau^{2^{i+2}l}u^{2^{i+1}k}u_i$ and $\tau^{2^{i+1}(2l+1)}u^{2^{i+1}k}h_i$ for $i < n$, and the classes $u^{\pm 2^{n}}$ and $\tau^{\pm 2^{n+1}}$. As the classes $\tau^{2^{i+2}l}u^{2^{i+1}k}u_i$ and $\tau^{2^{i+1}(2l+1)}u^{2^{i+1}k}h_i$ are $x^{2^n-1}$-torsion for $i < n$, having survived to the $E_{2^n}$-page they must be permanent cycles. It follows that the next differentials are determined by their effect on $u^{2^n}$ and $\tau^{2^{n+1}}$.

The differential $d_{2^{n+1}-1}(u^{2^n}) = u_n x^{2^{n+1}-1}$ follows from \cref{lem:cofib}. Now write $d_{2^{n+1}-1}(\tau^{2^{n+1}}) = \alpha \cdot x^{2^{n+1}-1}$. As $\xi$ is a permanent cycle, the Leibniz rule implies
\[
0 = d_{2^{n+1}-1}(\xi^{2^n}) = d_{2^{n+1}-1}(\rho^{2^n}\tau^{-2^{n+1}}u^{2^n}) = \rho^{2^n}(\tau^{-2^{n+2}}\alpha u^{2^n} + \tau^{-2^{n+1}}u_n)x^{2^{n+1}-1}.
\]
This is on the $E_{2^{n+1}-1} = E_{2^n}$-page, and so combines with our inductive hypothesis to imply
\[
z^{2^{n-1}}(u_n+\tau^{-2^{n+1}}u^{2^n}\alpha) \equiv 0 \pmod{u_0,\ldots,u_{n-1},z^{2^{n-1}}h_n}.
\]
This is only possible if
\[
u_n+\tau^{-2^{n+1}}u^{2^n}\alpha \equiv h_n\pmod{u_0,\ldots,u_{n-1}},
\]
and thus
\[
d_{2^{n+1}-1}(\tau^{2^{n+1}}) = \alpha x^{2^{n+1}-1} = \tau^{2^{n+1}}u^{-2^n}(h_n-u_n)x^{2^{n+1}-1}=w_nx^{2^{n+1}-1}
\]
as claimed.

To identify boundaries and cycles, observe that as a general property of the $x$-Bockstein spectral sequence, if we write $Z_r'[x]$ for its $r$-cycles then we have
\[
Z_r' = \ker\left(d_r\colon Z_{r-1}\rightarrow E_r\right) = \ker\left(d_r\colon Z_{r-1}\rightarrow E_r[x^{-1}]\right),
\]
i.e.\ to compute cycles it suffices to work in the $x$-inverted $x$-Bockstein spectral sequence, or equivalently the $x$-Bockstein spectral sequence with $x$ set to $1$. Our inductive hypothesis implies
\[
\frac{E_{2^n}}{(x-1)} \cong \frac{\left(E_0/(u_0,\ldots,u_{n-1})\right)[\rho,\xi,u^{\pm 2^n},\tau^{\pm 2^{n+1}}]_{(\rho,\xi)}^\wedge}{\left(\xi^{2^n}-\rho^{2^n}\tau^{-2^{n+1}}u^{2^n},\rho^{2^n}(u_n\tau^{2^{n+1}}+\rho^{2^{n+1}}h_{n+1}u^{2^n})\right)},
\]
and we have just produced the differentials
\[
d(u^{2^n}) = u_n,\qquad d(\tau^{2^{n+1}}) = \rho^{2^{n+1}}h_{n+1}.
\]
Thus $\ker(d)$ is generated over $E_0(\rho,\xi,u^{\pm 2^{n+1}},\tau^{\pm 2^{n+2}})_{(\rho,\xi)}^\wedge$ by $u_n\tau^{2^{n+1}} +  \rho^{2^{n+1}}h_{n+1}  u^{2^n}= \tau^{2^{n+1}}h_n$, and this leads to $Z_{2^{n+1}-1}' = Z_{2^{n+1}-1}$ as claimed. The identification of boundaries follows immediately.

The base case, concerning the $d_1$-differential and identification of the $E_2$-page, can be handled by considering $0 = d_1(\xi)$ just like the above, only taking into account the twist in the Leibniz rule for $d_1$ given in \cref{lem:cofib}. Alternately, one may just use the formula $d_1(a) = (a-\psi^{-1}(a))x$ given there, where the action of $\psi^{-1}$ is given in \cref{lem:bpp}.
\end{proof}

The ring $Z$ and ideal $B\subset Z[x]$ of the introduction may be identified as $Z = \bigcap_r Z_r$ and $B = \bigcup_r B_r$. Thus \cref{thm:ber} follows from \cref{thm:xbock} by letting $r \rightarrow\infty$.

\begin{rmk}
Although we have relied on the known computation of $\pi_\ast MPR$ in computing $\pi_\star b(MPR)$, this was not actually necessary: the proof of \cref{thm:xbock} gives an independent computation, as we now explain.

Note that no computation was needed to produce $x\in \pi_1 MP\bbR$ or prove $MP\bbR/(x)\simeq C_{2+}\otimes i_\ast MP$, as $x = a_\sigma \ol{u}$ where $\ol{u}$ generates the $n=1$ summand of $MP\bbR\simeq \bigoplus_{n\in\bbZ}\Sigma^{n(1+\sigma)}M\bbR$. Thus it suffices to describe the $x$-Bockstein spectral sequence $MP_0[u^{\pm 1},x]\Rightarrow MPR_\ast$. This is $MP_0$-linear by \cite[Proposition 2.27]{hukriz2001real}, which uses the theory of real orientations but not the computation of $\pi_\star M\bbR$. Thus it suffices to produce the differentials $d_{2^{n+1}-1}(u^{2^n}) = u_n x^{2^{n+1}-1}$. The differential $d_1(u) = 2x$ follows from the involution $\psi^{-1}(u) = -u$, so suppose inductively that we have computed $d_{2^{n+1}-1}(u^{2^n}) = u_n x^{2^{n+1}-1}$. 

Next note that no computation was needed in \cref{prop:xipc} to prove $\xi$ is a permanent cycle. The argument in \cref{thm:xbock} now applies to show $d_{2^{n+1}-1}(\tau^{2^{n+1}}) = -\rho^{2^{n+1}}h_{n+1} x^{2^{n+1}-1}$.

As in \cref{sec:compare}, there is a canonical map $q\colon b(MP^{\h C_2})\rightarrow F(EC_{2+},MP\bbR)$. Here, we write $MP^{\h C_2}$ and $F(EC_{2+},MP\bbR)$ instead of $MPR$ and $MP\bbR$ as the proof that $MP\bbR$ is Borel complete relies on knowledge of its $x$-Bockstein spectral sequence. The map $q$ fits into a diagram of cofiber sequences
\begin{center}\begin{tikzcd}[column sep=scriptsize]
\Sigma b(MP^{\h C_2})\ar[d,"q"]\ar[r,"x"]&b(MP^{\h C_2})\ar[r]\ar[d,"q"]&b(MP)\ar[d,"p"]\ar[r,"\partial"]&\Sigma^2 b(MP^{\h C_2})\ar[d,"q"]\\
\Sigma F(EC_{2+},MP\bbR)\ar[r,"x"]&F(EC_{2+},MP\bbR)\ar[r]&C_{2+}\otimes i_\ast MP\ar[r,"\partial'"]&\Sigma^2 F(EC_{2+},MP\bbR)
\end{tikzcd}.\end{center}
The $x$-Bockstein differential $d_{2^{n+1}-1}(\tau^{2^{n+1}}) = -\rho^{2^{n+1}}h_{n+1} x^{2^{n+1}-1}$ implies 
\[
\partial(\tau^{2^{n+1}}) = -\rho^{2^{n+1}}h_{n+1} x^{2^{n+1}-2}
\]
mod higher filtration, and as $p(\tau^2) = u_\sigma^2$ and $q(\rho) = a_\sigma$ it follows that
\begin{align*}
\partial'(u^{2^{n+1}}) &= \partial'(\ol{u}^{2^{n+1}}u_\sigma^{2^{n+1}}) = \ol{u}^{2^{n+1}}q(\partial(\tau^{2^{n+1}})) \\&
= \ol{u}^{2^{n+1}}q(-\rho^{2^{n+1}}h_{n+1} x^{2^{n+1}-2}) 
= \ol{u}^{2^{n+1}}a_\sigma^{2^{n+1}} u_{n+1} x^{2^{n+1}-2} = u_{n+1}x^{2^{n+2}-2}
\end{align*}
mod higher filtration. This gives the next $x$-Bockstein differential
\[
d_{2^{n+2}-1}(u^{2^{n+1}}) = u_{n+1} x^{2^{n+2}-1},
\]
completing the induction.
\tqed
\end{rmk}

We end this subsection with some observations about the structure of $\pi_\star b(ER)$.

\begin{prop}\label{cor:gap}
The $C_2$-spectrum $b(ER)$ has the following gap:
\[
\pi_{\ast\sigma-1}b(ER) = 0.
\]
\end{prop}
\begin{proof}
Declare the \textit{coweight} of a degree $c+w\sigma$ to be the quantity $c$, so that we are claiming $\pi_\star b(ER)$ vanishes in coweight $-1$. By \cref{thm:xbock}, $\pi_\star b(ER)$ is generated over the coweight $0$ classes $E_0(\rho,\xi)_{(\rho,\xi)}^\wedge$ by the class $x$ in coweight $1$ and the classes
\[
\tau^{2^{n+2}l}u^{2^{n+1}k}u_n,\qquad \tau^{2^{n+1}(2l+1)}u^{2^{n+1}k}h_n
\]
in coweights of the form $2^{n+1}t$. These classes are killed by $x^{2^{n+1}-1}$, and therefore cannot support long enough $x$-towers to reach coweight $-1$.
\end{proof}

\begin{prop}
If $E$ is $L_d$-local, then $b(ER)$ is $u^{\pm 2^{d+1}}$ and $\tau^{\pm 2^{d+1}}$-periodic. Moreover,
\[
x^{2^{d+1}-1} = 0,\qquad \rho^{2^d}x^{2^d-1} = 0,\qquad \xi^{2^d}x^{2^d-1} = 0
\]
in $Z[x]/B$.
\end{prop}
\begin{proof}
Recall that $E$ is $L_d$-local provided $u_d$ is invertible in $E_0/(u_0,\ldots,u_{d-1})$, or equivalently if the ideal $(u_0,\ldots,u_d)\subset E_0$ generates the entire ring. Thus as $u_i u^{\pm 2^{d+1}}$ is a permanent cycle for $i \leq d$, it follows that $u^{\pm 2^{d+1}}$ is also a permanent cycle. Likewise, as $u_i x^{2^{d+1}-1} = 0$ for $i\leq d$, it follows that $x^{2^{d+1}-1} = 0$.

Next, as $h_d(z) = u_d + O(z)$, it follows by Weierstrass preparation that $(u_0,\ldots,u_{d-1},h_d)\subset \pi_0 b(ER)\cong E^0 BC_2$ (see \cref{cor:integer}) generates the entire ring. As $u_i \tau^{\pm 2^{d+1}}$ for $i < d$ and $h_d \tau^{\pm 2^{d+1}}$ are permanent cycles, it follows that $\tau^{\pm 2^{d+1}}$ is a permanent cycle. Next, note that
\[
w_{d-1} = \rho^{2^d}h_d,\qquad \tau^{2^{d+1}}u^{2^d}w_{d-1} = \xi^{2^d}h_d.
\]
As $u_i x^{2^d-1} = 0$ for $i < d$, the identities $w_{d-1} x^{2^d-1} = 0$ and $\tau^{2^{d+1}}u^{2^d}w_{d-1}x^{2^d-1} = 0$ then imply $\rho^{2^d}x^{2^d-1} = 0$ and  $\xi^{2^d}x^{2^d-1} = 0$.
\end{proof}

\begin{rmk}\label{rmk:table}
The following tables may be helpful in getting acquainted with the general shape of $\pi_\star b(ER)$, and especially for visualizing the arguments in \cref{sec:tr} and \cref{sec:ext}:
\renewcommand{\arraystretch}{1.1}
\[
\arraycolsep=2pt
\begin{array}{|c|c|c|c|c|c|c|c|c|c|}
\hline
u_0,u_1 &\xi u_1,\xi h_2 &\xi^2u_1,\xi^2h_2 &\xi^3u_1,\xi^3h_2 &u_0,\xi^4h_2 &\rho^3u_1,\xi^5h_2 &\rho^2u_1,\xi^6h_2 &\rho u_1,\xi^7h_2 &u_0,u_1\\
\rho^8x^8,h_2 &\rho^7x^8 &\rho^6x^8 &\rho^5x^8 &\rho^4 x^8 &\rho^3x^8 &\rho^2x^8 &\rho x^8 &x^8\\
\hline

\rho^7x^7 &\rho^6x^7 &\rho^5x^7 &\rho^4x^7 &\rho^3x^7 &\rho^2x^7 &\rho x^7 &x^7 &\xi x^7\\
\hline

h&&h_1x^2&\xi h_1x^2&h&&&&h\\
\rho^6 x^6 & \rho^5x^6 &\rho^4x^6 &\rho^3x^6 &\rho^2x^6 &\rho x^6 &x^6 &\xi x^6&\xi^2x^6\\
\hline

&h_1 x& \xi h_1 x& & & & & & \rho h_1 x \\
\rho^5 x^5& \rho^4 x^5& \rho^3 x^5& \rho^2 x^5& \rho x^5& x^5& \xi x^5& \xi^2 x^5& \xi^3 x^5\\
\hline

u_0,h_1&\xi h_1&\xi^2h_1&\xi^3h_1&u_0&\rho^3 h_1&\rho^2h_1&\rho h_1&u_0,h_1\\
\rho^4x^4& \rho^3 x^4& \rho^2 x^4& \rho x^4& x^4& \xi x^4& \xi^2 x^4& \xi^3 x^4& \xi^4 x^4\\
\hline

\rho^3 x^3& \rho^2 x^3& \rho x^3& x^3& \xi x^3& \xi^2x^3& \xi^3x^3& \xi^4x^3& \xi^5 x^3\\
\hline

h & & & & h & & & & h \\
\rho^2x^2 & \rho x^2 & x^2 & \xi x^2 & \xi^2 x^2 & \xi^3 x^2 & \xi^4 x^2 & \xi^5 x^2 & \xi^6 x^2 \\
\hline

&&&&&&&&\rho u_1 x \\
\rho x & x & \xi x & \xi^2 x & \xi^3 x & \xi^4 x & \xi^5 x & \xi^6 x & \xi^7 x\\
\hline

&&&& u_0 & \rho^3 u_1 & \rho^2 u_1  & \rho u_1  & u_0,u_1\\
1 & \xi & \xi^2 & \xi^3 & \xi^4 & \xi^5 & \xi^6 & \xi^7 & \xi^8\\
\hline
\end{array}
\]
\[
\arraycolsep=1pt
\begin{array}{|c|c|c|c|c|c|c|c|c|c|}
\hline
\xi u_1,\rho^7h_2&\xi^2u_1,\rho^6 h_2& \xi^3 u_1,\rho^5 h_2 &u_0, \rho^4h_2& \rho^3u_1,\rho^3h_2& \rho^2u_1, \rho^2h_2&\rho u_1, \rho h_2&u_0,u_1\\
\xi x^8&\xi^2 x^7 & \xi^3 x^7 & \xi^4 x^8 & \xi^5 x^7 & \xi^6 x^8 & \xi^7 x^8 & \xi^8 x^8, h_2\\
\hline

\xi^2x^7&\xi^3 x^7 & \xi^4 x^7 & \xi^5 x^7 & \xi^6 x^7 & \xi^7 x^7 & \xi^8 x^7 & \xi^9 x^7\\
\hline

\rho h_1 x^2&x^2 h_1 & x^2 \xi h_1 & h &&&& h \\
\xi^3x^6&\xi^4 x^6 & \xi^5 x^6 & \xi^6 x^6 & \xi^7 x^6 & \xi^8 x^6 & \xi^9 x^6 & \xi^{10}x^6 \\
\hline

h_1x&x \xi h_1 & & &&&& x \rho h_1\\
\xi^4x^5&\xi^5 x^5 & \xi^6 x^5 & \xi^7 x^5 & \xi^8 x^5 & \xi^9 x^5 & \xi^{10}x^5 & \xi^{11}x^5 \\
\hline

\xi h_1&\xi^2 h_1 & \xi^3 h_1&u_0&\rho^3 h_1&\rho^2 h_1&\rho h_1&u_0, h_1\\
\xi^5x^4&\xi^6 x^4 & \xi^7 x^4 & \xi^8 x^4 & \xi^9 x^4 & \xi^{10}x^4 & \xi^{11}x^4 & \xi^{12}x^4 \\
\hline

&&&&&&&x^3\rho^3 u_2 \\
\xi^6x^3&\xi^7 x^3 & \xi^8 x^3 & \xi^9 x^3 & \xi^{10}x^3 & \xi^{11}x^3 & \xi^{12}x^3 & \xi^{13}x^3\\
\hline

\rho u_1 x^2&x^2 u_1 & x^2\xi  u_1 &x^2\xi^2 u_1 \text{/} x^2\rho^6 u_2 
& x^2\rho^5 u_2 & x^2 \rho^4 u_2 & x^2 \rho^3 u_2 & h, x^2\rho^2 u_2 \\
\xi^7x^2&\xi^8 x^2 & \xi^9 x^2 & h, \xi^{10}x^2  & \xi^{11}x^2 & \xi^{12}x^2 & \xi^{13}x^2 & \xi^{14}x^2\\
\hline

u_1x&x\xi u_1 & x \xi^2 u_1\text{/}x\rho^6 u_2 
& x\rho^5 u_2& x\rho^4 u_2 & x \rho^3  u_2 & x \rho^2  u_2 & x \rho  u_1, x \rho u_2 \\
\xi^8x&\xi^9 x & \xi^{10}x & \xi^{11}x & \xi^{12}x & \xi^{13}x & \xi^{14}x & \xi^{15}x \\
\hline

\xi u_1,\rho^7 u_2&\xi^2 u_1 , \rho^6 u_2 & \xi^3 u_1 , \rho^5 u_2 & u_0, \rho^4 u_2& \rho^3 u_1, \rho^3 u_2 & \rho^2 u_1, \rho^2 u_2 & \rho  u_1, \rho u_2 & u_0 ,  u_1\\
\xi^9&\xi^{10}&\xi^{11}&\xi^{12}&\xi^{13}&\xi^{14}&\xi^{15}&\xi^{16}, u_2\\
\hline
\end{array}
\]
\renewcommand{\arraystretch}{1}%
These describe generators of $\pi_{c+w\sigma}b(ER)$ as a module over $\pi_0 b(ER)\cong E^0 BC_2$ in coweight $0\leq c\leq 8$ and stem $0\leq c+w\leq 16$, the first table containing stems $0\leq c+w\leq 8$ and second $9\leq c+w\leq 16$. It is arranged by stem and coweight: the box at coordinate $(s,c)$ contains a list of generators for $\pi_{c+(s-c)\sigma}b(ER)$. For space reasons, we have omitted any $\tau^{2i}u^j$ terms. These may recovered by comparing degrees: for example, the box in coordinate $(8,5)$ has entries $\rho h_1 x$ and $\xi^3 x^5$, and this means that $\pi_{5+3\sigma}b(ER)$ is generated over $\pi_0 b(ER)$ by $\rho\cdot \tau^{-4}u^4h_1\cdot x$ and $\xi^3 x^5$.

The entry $x\xi^2u_1/x\rho^6u_2$ indicates that either $x\xi^2 u_1$ or $x \rho^6 u_2$ may be chosen as a generator, and likewise for $x^2\xi^2 u_1/x^2\rho^6 u_2$. This sort of choice also appears on the $0$-line: for example, in box $(5,0)$ one could replace $\rho^3 u_1$ with $\xi u_0$.

These tables assume that $E$ has sufficiently large height, say $E = MP$.
\tqed
\end{rmk}

\section{Transfers}\label{sec:tr}

Recall that there are cofiber sequences
\begin{equation}\label{eq:cofibs2}\begin{tikzcd}[row sep=tiny]
\Sigma^{-\sigma}b(ER)\ar[r,"\rho"]&b(ER)\ar[r]&C_{2+}\otimes i_\ast ER \\
\Sigma^\sigma b(ER)\ar[r,"\xi"]&b(ER)\ar[r]&E\bbR
\end{tikzcd}\end{equation}
of $C_2$-spectra. The first is a general cofiber sequence that exists for any $C_2$-spectrum, given that $C_{2+}\otimes b(ER)\simeq C_{2+}\otimes i_\ast ER$, and the second was shown in \cref{prop:xicof}. Here,
\[
\pi_\star(C_{2+}\otimes i_\ast ER) \cong \pi_\ast ER[u_\sigma^{\pm 1}],\qquad |u_\sigma| = 1-\sigma,
\]
and $\pi_\star E\bbR$ was described in \cref{sec:even}.

Associated to the cofiber sequences of \cref{eq:cofibs2} are boundary maps
\[
\tr(u_\sigma^{-1}\cdot\bs)\colon \pi_{\star+1-\sigma}(C_{2+}\otimes i_\ast ER)\rightarrow \pi_{\star} b(ER),\qquad \partial\colon \pi_{\star+1+\sigma} E\bbR\rightarrow \pi_\star b(ER).
\]
The first of these is the transfer for the $C_2$-spectrum $b(ER)$. Both are $\pi_\star b(ER)$-linear.

\begin{prop}\label{prop:trp}
The above transfer and boundary maps satisfy
\begin{align*}
\tr(u_\sigma^{-1}\cdot u_\sigma^{2^n(2k+1)}) &= \rho^{2^n-1}\tau^{2^{n+1}k}h_n x^{2^n-1}+O(x^{2^n}),\\
\partial(\ol{u}^{2^n(2k+1)}) &= \xi^{2^n-1}\tau^{-2^{n+1}k}u^{2^{n+1}k}h_nx^{2^n-1}+O(x^{2^n})
\end{align*}
for $n\geq 0$ and $k\in\bbZ$.
\end{prop}
\begin{proof}
The error terms are present just because $\tau^{2^{n+1}k}h_n$ and $\tau^{-2^{n+1}k}u^{2^{n+1}k}h_n$ have only been defined mod $x$, so we omit them in the proof.

First consider the case $n=0$. These claimed values are not hidden in the $x$-Bockstein spectral sequence, so it suffices to show that they hold after coning off $x$. After coning off $x$, the cofiber sequences \cref{eq:cofibs2} take the form
\begin{equation}\begin{tikzcd}[row sep=tiny]
\Sigma^{-\sigma}b(E)\ar[r,"\rho"]&b(E)\ar[r]&C_{2+}\otimes i_\ast E \\
\Sigma^\sigma b(E)\ar[r,"\rho\tau^{-2}u"]&b(E)\ar[r]&C_{2+}\otimes i_\ast E 
\end{tikzcd}.\end{equation}
In particular, $\partial(\ol{u}\alpha) = \tr(\alpha)$. By \cite[Remark 6.15]{hopkinskuhnravenel2000generalized}, the transfer $\tr\colon \pi_0 E \rightarrow E^0 BC_2 \cong \pi_0 b(E)$ satisfies $\tr(1) = h$. The proof is to observe that $h$ is the unique class which satisfies
\[
\rho\cdot h = 0,\qquad h\equiv 2 \pmod{\rho}.
\]
As $\tr$ and $\partial$ are $\pi_\star b(E)$-linear, we deduce
\[
\tr(u_\sigma^{-1}\cdot u_\sigma^{2n+1}) = \tau^{2n}\tr(1) = \tau^{2n}h,\qquad \partial(\ol{u}^{2n+1}) = \tau^{-2n}u^{2n}\partial(\ol{u}) = \tau^{-2n}u^{2n}h
\]
as claimed. The argument is essentially the same for $n\geq 1$. Observe that
\[
\rho^{2^n}\tau^{2^{n+1}k}h_n = \tau^{2^{n+1}k}w_{n-1},\qquad \xi^{2^n}\tau^{-2^{n+1}k}h_n = \tau^{-2^{n+1}(k+1)}u^{2^n}w_{n-1}.
\]
In particular $\rho^{2^n-1}\tau^{2^{n+1}k}h_nx^{2^n-1}$ and $\xi^{2^n-1}\tau^{-2^{n+1}k}u^{2^{n+1}k}h_nx^{2^n-1}$ generate the kernels of $\rho$ and $\xi$ in their respective degrees. As the kernels of $\rho$ and $\xi$ are generated by the images of $\tr$ and $\partial$, this gives the claimed values of $\tr$ and $\partial$ up to multiplication by a unit, which may then be ruled out by working in the universal case $E = MP$. 
\end{proof}

\pagebreak[1]

\section{Hidden extensions}\label{sec:ext}

We now turn our attention to hidden extensions. We begin with a general discussion.  Write $Z[x]/B$ for the $x$-adic associated graded of $\pi_\star b(ER)$, as computed in \cref{sec:bock}. In general, hidden extensions in the $x$-Bockstein spectral sequence arise from the failure of $\pi_\star b(ER)$ to be isomorphic to $Z[x]/B$, and especially for relations to fail to lift through the map
\begin{equation}\label{eq:edge}
\pi_\star b(ER) \rightarrow (\pi_\star b(ER))/(x)\cong Z \subset\pi_\star b(E).
\end{equation}
Recall that
\[
\pi_\star b(E) = \frac{E_0[\rho,\tau^{\pm 2},u^{\pm 1}]_\rho^\wedge}{(\rho\cdot h)}.
\]
This indicates that the simple indecomposable hidden extensions will be those $\rho$ and $\xi$-extensions lifting relations of the form
\begin{equation}\label{eq:rhoh}
\rho\cdot \tau^{2i} u^{j} h = 0,\qquad \xi\cdot \tau^{2i} u^{j} h = 0,
\end{equation}
where $i$ and $j$ are such that $\tau^{2i}u^j h \in Z$.

If a relation of this sort lifts to $\pi_\star b(ER)$, then necessarily the corresponding $\tau^{2i}u^jh$ is in the image of the transfer or boundary studied in the previous section. These classes are generally not in the image of the transfer or boundary, and so one knows from the start that the relations in \cref{eq:rhoh} generally lift to nontrivial hidden extensions in $\pi_\star b(ER)$.

One can use \cref{prop:trp} to compute some of these directly:
\[
\xi\cdot \tau^{2^{n+1}(2k+1)}h = \xi \cdot \tr(u_\sigma^{2^{n+1}(2k+1)}) = \tr(x u_\sigma^{-1}\cdot u_\sigma^{2^{n+1}(2k+1)}) = \rho^{2^{n+1}-1}\tau^{2^{n+1}k}h_{n+1}x^{2^{n+1}}
\]
by Frobenius reciprocity, and likewise
\[
\rho\cdot\tau^{-2^{n+1}(2k+1)}u^{2^{n+1}(2k+1)}h = \rho \cdot \partial(\ol{u}^{2^{n+1}(2k+1)+1}) = \xi^{2^{n+1}-1}\tau^{2^{n+2}k}u^{2^{n+2}k}h_{n+1}x^{2^{n+1}}.
\]
In general however, a more indirect approach is necessary. Consider the cofiber sequences
\begin{equation}\begin{tikzcd}[row sep=tiny]\label{eq:cofibs}
\Sigma^{-\sigma}b(ER)\ar[r,"\rho"]&b(ER)\ar[r,"p"]&C_{2+}\otimes i_\ast ER \\
\Sigma^\sigma b(ER)\ar[r,"\xi"]&b(ER)\ar[r,"q"]&E\bbR
\end{tikzcd}.\end{equation}
The long exact sequences associated to these imply that the image of $\rho$ is equal to the kernel of the forgetful map $p\colon\pi_{\star}b(ER)\rightarrow \pi_{|\star|}ER$, and that the image of $\xi$ is equal to the kernel of the canonical map $q\colon\pi_\star b(ER)\rightarrow \pi_\star E\bbR$. To find elements of these kernels, one looks for elements in $\pi_\star b(ER)$ that lift the relations $u_n x^{2^{n+1}-1} = 0$. This relation already holds in $\pi_\star b(ER)$, so we need only consider lifts involving the filtration-shifting identities $p(\xi) = x$ and $q(\rho) = a_\sigma$. In this way we focus our attention on those classes of the form
\begin{equation}\label{eq:imrhoxi}
\tau^{2^{n+2}k}u^{2^{n+2}l}h_{n+1}\xi^rx^s,\qquad \tau^{2^{n+2}k}u^{2^{n+2}l}h_{n+1}\rho^rx^s,
\end{equation}
where $r+s\geq 2^{n+1}-1$ and $r\geq 1$ and $s < 2^{n+2}-1$. By the preceding discussion, the former must be in the image of $\rho$ and the latter in the image of $\xi$, and when this is not the case in $Z[x]/B$ there must be a hidden extension making it so. If $r+s > 2^{n+2}-1$, then the witness to the classes in \cref{eq:imrhoxi} being in the image of $\rho$ or $\xi$ may be obtained by multiplying a smaller witness with some suitable power of $\rho$ or $\xi$ and $x$. Thus we are led to focus on the case where $r+s=2^{n+2}-1$. We will show that when $s$ is even, the necessary hidden extensions are exactly those lifting the relations in \cref{eq:rhoh}. First, a couple observations.

\begin{lemma}\label{lem:ext2}
Fix positive integers $r+s=2^{n+2}-1$ with $s$ even. Then the classes
\[
\tau^{2^{n+2}k}u^{2^{n+2}l}h_{n+1}\xi^rx^s,\qquad \tau^{2^{n+2}k}u^{2^{n+2}l}h_{n+1}\rho^rx^s
\]
are not in the image of $\rho$ or $\xi$ respectively in $Z[x]/B$, at least when $E = MP$.
\end{lemma}
\begin{proof}
Consider the first case. Suppose towards contradiction that
\[
\tau^{2^{n+2}k}u^{2^{n+2}l}h_{n+1}\xi^rx^s = \rho \alpha x^s
\]
for some $\alpha \in Z$. As the $x$-Bockstein spectral sequence has only odd differentials and $s$ is even, necessarily we can divide out by $x$ to obtain
\begin{equation}\label{eq:badrel}
(\tau^{2^{n+2}k}u^{2^{n+2}l}h_{n+1}\xi^r-\rho \alpha)x^{s-1} = 0.
\end{equation}
This means that $\tau^{2^{n+2}k}u^{2^{n+2}l}h_{n+1}\xi^r-\rho \alpha$ detects some class $\theta \in\pi_\star b(MPR)$ satisfying $\theta\cdot x^{s-1} = 0$. Write $p\colon \pi_\star b(MPR)\rightarrow \pi_{|\star|}MPR$ for the restriction. As $p(\xi) = x$, necessarily $p(\theta)$ is detected by $u^{2^{n+2}l}u_{n+1}x^r$. Thus
\[
0 = p(\theta\cdot x^{s-1})\equiv u^{2^{n+2}l}u_{n+1}x^{r+s-1}\pmod{x^{r+s}}
\]
in $\pi_\ast MPR$. As $r+s-1 < 2^{n+2}-1$, this is incompatible with the structure of the $x$-Bockstein spectral sequence for $\pi_\ast MPR$, a contradiction. The second case is identical, only instead using the map $b(MPR)\rightarrow MP\bbR$ in place of the restriction.
\end{proof}

\begin{lemma}\label{lem:ext1}
Suppose that $i$ and $j$ are such that $\tau^{2i}u^j h \in Z$. Then $\tau^{2i}u^j h$ generates the kernels of $\rho$ and $\xi$ in its degree of $Z[x]/B$ as a module over $\pi_0 b(ER)$.
\end{lemma}
\begin{proof}
The class $\tau^{2i}u^jh$ generates the kernels of $\rho$ and $\xi$ in $Z$, as this is the case in $\pi_\star b(E)$. Thus the lemma follows from the following observation: $Z[x]/B$ contains no $x$-divisible elements in the kernel of $\rho$ or $\xi$ in even degrees, that is in degrees of the form $c+w\sigma$ with both $c$ and $w$ even. Indeed, any $x$-divisible element in even degree and in a given filtration must be of the form $\alpha x^{2r}=0$ with $\alpha \in Z$ in even degree. As $\alpha$ is in even degree and $B$ is generated by classes of the form $w\cdot x^?$ with $w$ in even degree, relations $\rho\alpha x^{2r} = 0$ or $\xi\alpha x^{2r}=0$ are only possible if $\alpha x^{2r} = 0$ already, proving the lemma.
\end{proof}

We may now give the main theorem of this subsection.

\begin{theorem}\label{thm:rhoxiext}
There are extensions
\begin{align*}
\rho\cdot \tau^{2(2^{n+1}k-r)}u^{2^{n+1}(2l+1)}h &= \left(\tau^{2^{n+2}k}u^{2^{n+2}l}h_{n+1} \xi^{2r-1}+O(\rho)\right)x^{2^{n+2}-2r}+O(x^{2^{n+2}-2r+1})\\
\xi\cdot\tau^{2(2^{n+1}k+r)}u^{2(2^n(2l+1)-r)}h &= \left(\tau^{2^{n+2}k}u^{2^{n+2}l}h_{n+1} \rho^{2r-1}+O(\xi)\right)x^{2^{n+2}-2r}+O(x^{2^{n+2}-2r+1})
\end{align*}
for $n\geq 0$, $k,l\in\bbZ$, and $1\leq r \leq 2^{n+1}-1$.
\end{theorem}
\begin{proof}
It suffices to produce these extensions in the universal case $E = MP$. This ensures that the terms on the right are nonzero, so that these are nontrivial extensions. As discussed above, the cofiber sequences of \cref{eq:cofibs} show that the terms 
\[
\tau^{2^{n+2}k}u^{2^{n+2}l}h_{n+1} \xi^{2r-1}x^{2^{n+2}-2r},\qquad \tau^{2^{n+2}k}u^{2^{n+2}l}h_{n+1} \rho^{2r-1}x^{2^{n+2}-2r}
\]
must be in the image of $\rho$ and $\xi$ respectively. By \cref{lem:ext2}, this is not the case in $Z[x]/B$, so there must be hidden extensions making it so. In other words, there must be hidden extensions of the form
\begin{align*}
\rho\cdot \alpha &= \left(\tau^{2^{n+2}k}u^{2^{n+2}l}h_{n+1} \xi^{2r-1}+O(\rho)\right)x^{2^{n+2}-2r}+O(x^{2^{n+2}-2r+1}),\\
\xi\cdot \beta &= \left(\tau^{2^{n+2}k}u^{2^{n+2}l}h_{n+1} \rho^{2r-1}+O(\xi)\right)x^{2^{n+2}-2r}+O(x^{2^{n+2}-2r+1}),
\end{align*}
where $\alpha$ and $\beta$ are detected by classes in $Z[x]/B$ killed by $\rho$ and $\xi$ respectively. The error terms ensure that we do not need to pin down $\alpha$ and $\beta$ precisely, but only the $\pi_0 b(MPR)$-submodule of $Z[x]/B$ that they generate. By \cref{lem:ext1}, the extensions given in the theorem statement are the only possibilities in these degrees.
\end{proof}

\begin{rmk}
This leaves open the problem of finding witnesses to the classes of \cref{eq:imrhoxi} being in the image of $\rho$ or $\xi$ in the case where $r+s=2^{n+2}-1$ and $r$ is even. In some cases no hidden extension is necessary, for example
\begin{align*}
\rho^{2^{n+1}}h_{n+1}x^{2^{n+1}-1} &= w_n x^{2^{n+1}-1} = 0,\\
\xi^{2^{n+1}}h_{n+1}x^{2^{n+1}-1} &= \tau^{-2^{n+2}}u^{2^{n+1}}w_n x^{2^{n+1}-1} = 0.
\end{align*}
However, the general situation seems to be rather subtle. For example, for $h_2 \xi^2 x^5$ to be in the image of $\rho$, the only possibility is that $\rho \tau^{-4}u^4 h_1 x$ detects a class satisfying
\[
\rho\cdot \rho\tau^{-4}u^4h_1 x = h_2 \xi^2 x^5 + O(\rho).
\]
On the other hand, we have
\[
\rho^2\tau^{-4}u^4h_1 = \tau^{-4}u^4 w_0,\qquad \tau^{-4}u^4 w_0\cdot x = 0
\]
in $Z[x]/B$. This indicates the existence of a mixed extension along the lines of
\[
\rho^2\tau^{-4}u^4 h_1 = \tau^{-4}u^4 w_0 + h_1 \xi^2 x^4+O(\rho).
\]
Note that if $\theta \in \pi_\star b(ER)$ is detected by $\tau^{-4}u^4 w_0$, then so is $\theta + h_1 \xi^2 x^4$. Thus for such an extension to even be defined, one must specify some information about how one lifts elements from $Z$ to $\pi_\star b(ER)$, and these considerations are outside the scope of our investigation.
\tqed
\end{rmk}

\section{Some Mahowald invariants}\label{sec:root}

We end by giving some examples of computations within the ring $\pi_\star b(ER)$. Our examples will center around the following definition.

\begin{defn}\label{def:root}
Given a spectrum $A$, the \textit{$A$-based Mahowald invariant} is a partially defined and multivalued function
\[
R_A\colon \pi_\ast A \rightharpoonup \pi_\ast A,
\]
i.e.\ a relation on $\pi_\ast A$, defined as follows: given $y\in \pi_n A$ and $z\in \pi_{n+k}A$, we say $z\in R_A(y)$ if $z$ lifts to a class $\zeta\in \pi_\star b(A)$ such that $\rho^N y = \rho^{N+k}\zeta$ for $N\gg 0$, and moreover $k$ is as large as possible.
\tqed
\end{defn}

\begin{rmk}
There are natural maps $\pi_n A \rightarrow \pi_n A^{t C_2}$ and $\pi_{c+w\sigma} b(A)\rightarrow \pi_c A^{t C_2}$, and the condition $\rho^N y = \rho^{N+k}\zeta$ for $N\gg 0$ amounts to asking that $y = \zeta$ in $\pi_\ast A^{t C_2}$. When $A = S$, this construction recovers the classical Mahowald invariant, commonly called the root invariant. See \cite{mahowaldravenel1993root} for additional background, \cite{brunergreenlees1995bredon} for the relation to $C_2$-equivariant homotopy theory, which connects \cref{def:root} to other definitions, \cite{behrens2007some} for the state of the art in $S$-based Mahowald invariants at the prime $2$, \cite{quigley2019tmf} for further discussion of $A$-based Mahowald invariants with $A\neq S$, and \cite{lilormanquigley2022tate} for more information about spectra related to $ER^{t C_2}$.
\tqed
\end{rmk}

In \cite{lishiwangxu2019hurewicz}, Li--Shi--Wang--Xu prove that the Hurewicz image of real bordism detects the Hopf elements, Kervaire classes, and $\kappabar$ family. These are the elements in $\pi_\ast S$ detected on the $E_2$-page of the Adams spectral sequence by the classes $h_i$, $h_j^2$, and $g_{k+1}$ respectively; note there is no claimed relation between $h_i$ here and the elements $h_i$ in $\pi_\star b(E)$. These classes arrange into $\Sq^0$ families, i.e.\
\begin{equation}\label{eq:sq0}
\Sq^0(h_i) = h_{i+1},\qquad \Sq^0(h_j^2) = h_{j+1}^2,\qquad \Sq^0(g_{k+1}) = g_{k+2}.
\end{equation}
Informally, this means that they arise as iterated Mahowald invariants at the level of $\Ext$. Of course this cannot lift to the level of homotopy, as not all of these classes are permanent cycles; still, it is known that $\eta \in R_S(2)$, $\nu \in R_S(\eta)$, and $\sigma \in R_S(\nu)$, and it is conjectured that $\theta_{j+1} \in R_S(\theta_j)$ for $j\geq 3$ provided $\theta_{j+1}$ exists, see \cite[Proposition 2.4]{mahowaldravenel1993root}.

We can compute the iterated $MPR$-based Mahowald invariants of the classes $2$, $\theta_0=4$, and $\kappabar$, yielding an analogue of \cref{eq:sq0}. Our computation works just as well for $ER$ in a range depending on the height of $E$. First we need to know how $\kappabar$ sits inside $\pi_\star MPR$.

\begin{lemma}\label{lem:kappabar}
The class $\kappabar$ is detected by $MPR$, with Hurewicz image $u_2^4u^8x^4$.
\end{lemma}
\begin{proof}
If $\kappabar$ is detected by $MPR$, then it is detected by $MR$. As $\pi_{20}MR = \bbZ/(2)\{u_2^4 u^8 x^4\}\subset\pi_{20}MPR$, it suffices just to show that $\kappabar$ is detected by $MR$, which was shown in \cite{lishiwangxu2019hurewicz}. Alternately, as there is a ring map $MR\rightarrow \TMF_0(3)$ \cite{hillmeier2017c2}, it suffices to show that $\kappabar$ is detected in the latter, and here one may appeal to \cite{mahowaldrezk2009topological}.
\end{proof}

We now abbreviate $R = R_{MPR}$.

\begin{theorem}
Define elements
\[
a_n = u_n x^{2^n-1}\in \pi_{2^n-1}MPR,\qquad b_m = u_{m+1}^4 u^{2^{m+2}}x^{2^{m+2}-4}\in\pi_{4(3\cdot 2^m-1)}MPR
\]
for $n\geq 0$ and $m\geq 1$, so that for example $a_0 = 2$ and $b_1 = \kappabar$. Then there are $MPR$-based Mahowald invariants
\[
a_{n+1} \in R(a_n),\qquad a_{n+1}^2 \in R(a_n^2),\qquad b_{m+1} \in R(b_m).
\]
\end{theorem}
\begin{proof}
First consider $a_n$. As
\[
h_n\equiv u_n + \rho^{2^n}\xi^{2^n}h_{n+1}\pmod{u_0,\ldots,u_{n-1}},
\]
the relation $\rho^{2^n}h_n\cdot x^{2^n-1} = 0$ implies
\begin{equation}\label{eq:rp}
\rho^{2^n}\cdot u_n x^{2^n-1} = - \rho^{2^{n+1}}\cdot h_{n+1}\xi^{2^n}x^{2^n-1}.
\end{equation}
There are no further relations and $h_{n+1}\xi^{2^n}x^{2^n-1}$ lifts $u_{n+1}x^{2^{n+1}-1}=a_{n+1}$, yielding $a_{n+1} \in R(a_n)$. The case of $a_n^2$ is identical, only we must apply \cref{eq:rp} twice:
\[
\rho^{2^n}\cdot u_n^2 x^{2(2^n-1)} = -\rho^{2^{n+1}}\cdot h_{n+1}\xi^{2^n}u_nx^{2(2^n-1)} = \rho^{3\cdot 2^n}\cdot h_{n+1}^2 \xi^{2^{n+1}}x^{2(2^n-1)}.
\]

Now consider $b_m$. As $2^{m+2}-4\geq 2^{m+1}-1$ for $m\geq 1$, we may apply \cref{eq:rp} thrice to obtain
\begin{equation}\label{eq:rp2}
\rho^{2^{m+1}}\cdot u_{m+1}^4u^{2^{m+2}}x^{2^{m+2}-4} = \rho^{2^{m+3}}\cdot u_{m+1}u^{2^{m+2}}\cdot \xi^{3\cdot 2^{m+1}} h_{m+2}^3 x^{2^{m+2}-4}.
\end{equation}
At this point additional care is needed: we cannot apply \cref{eq:rp} again, as despite appearances $u_{m+1}u^{2^{m+2}}$ is indecomposable. Instead, the relation $\rho \cdot h = 0$ gives
\[
0 \equiv u_{m+1}\rho^{2^{m+2}-1}\tau^{-2^{m+2}+2}u^{2^{m+1}-1}+h_{m+2}\rho^{2^{m+3}-1}\tau^{-2^{m+3}+2}u^{2^{m+2}-1}\pmod{u_0,\ldots,u_m}
\]
in $\pi_\star b(MP)$, and thus
\begin{align*}
u_{m+1}u^{2^{m+2}}\cdot \xi^{2^{m+2}-1} &= u_{m+1}u^{2^{m+2}}\cdot \rho^{2^{m+2}-1}\tau^{-2^{m+3}-2}u^{2^{m+2}-1}\\
&\equiv \tau^{-2^{m+3}}u^{2^{m+3}}h_{m+3}\cdot \rho^{3\cdot 2^{m+1}}\xi^{2^{m+1}-1}\pmod{u_0,\ldots,u_m}.
\end{align*}
Substituting this into \cref{eq:rp2} yields
\[
\rho^{2^{m+1}}\cdot u_{m+1}^4u^{2^{m+2}}x^{2^{m+2}-4} = \rho^{7\cdot 2^{m+1}}\cdot \tau^{-2^{m+3}}u^{2^{m+3}}h_{m+2}^4\cdot \xi^{2^{m+2}}x^{2^{m+2}-4}.
\]
We cannot pull this class back any further. Thus, as $\tau^{-2^{m+3}}u^{2^{m+3}}h_{m+2}^4\cdot \xi^{2^{m+2}}x^{2^{m+2}-4}$ lifts $u_{m+2}^4u^{2^{m+3}}x^{2^{m+3}-4}=b_{m+1}$, we obtain $b_{m+1} \in R(b_m)$.
\end{proof}

\begingroup
\raggedright
\bibliography{refs}
\bibliographystyle{alpha}
\endgroup
\end{document}